\documentclass{article}
\usepackage[]{geometry}
\usepackage{stmaryrd}
\usepackage{mathrsfs}
\usepackage{amsfonts}
\usepackage{amssymb}
\usepackage{amsmath}
\usepackage{enumerate}
\usepackage{latexsym}
\usepackage{multirow}
\usepackage[authoryear]{natbib}
\usepackage[pdftex]{color,graphicx}
\newcommand{\function}[2]{#1\left(#2\right)}
\newcommand{\p}[1]{\mathbf{P}\left(#1\right)}

\newenvironment{proof}{\textbf{Proof: }}{$\hfill \square \\$}

\newtheorem{lemma}{Lemma}[section]

\newtheorem{thm}{Theorem}[section]
\newtheorem{defn}{Definition}[section]
\newcommand{\ind}[1]{\mathbf{1}_{#1}}

\newcommand{\argmax}{\operatornamewithlimits{\textrm{argmax}}}

\newcommand{\lsup}{\operatornamewithlimits{\overline{\lim}}}
\newcommand{\linf}{\operatornamewithlimits{\underline{\lim}}}
\newcommand{\sargmax}{\operatornamewithlimits{\textrm{sargmax}}}
\newcommand{\largmax}{\operatornamewithlimits{\textrm{largmax}}}

\usepackage[colorlinks,citecolor=blue,urlcolor=blue,linkcolor=magenta]{hyperref}

\begin{document}
\title{\textsc{A Continuous Mapping Theorem for the Smallest Argmax Functional}}
\author{\begin{tabular}{c}
Emilio Seijo and Bodhisattva Sen\\
Columbia University
\end{tabular}}
\date{}
\maketitle
\begin{abstract}
This paper introduces a version of the argmax continuous mapping theorem that applies to M-estimation problems in which the objective functions converge to a limiting process with multiple maximizers. The concept of the smallest maximizer of a function in the $d$-dimensional Skorohod space is introduced and its main properties are studied. The resulting continuous mapping theorem is applied to three problems arising in change-point regression analysis. Some of the results proved in connection to the $d$-dimensional Skorohod space are also of independent interest.
\end{abstract}

\tableofcontents


\tableofcontents

\section{Introduction}\label{s1}
Many estimators in statistics are defined as the maximizers of certain stochastic processes, called objective functions. This procedure for computing estimators is known as M-estimation and is quite common in modern statistics. A standard way to find the asymptotic distribution of a given M-estimator, is to obtain the limiting law of the (appropriately normalized) objective function and then apply the so-called argmax continuous mapping theorem (see Theorem 3.2.2, page 286 of \cite{vw} for a quite general version of this result). Chapter 3.2 in \cite{vw} gives an excellent account of M-estimation problems and applications of the argmax continuous mapping theorem.

Despite its proven usefulness in a wide range of applications, there are some M-estimation problems that cannot be solved by an application of the usual argmax continuous mapping theorem. This is particularly true when the objective functions converge in distribution to the law of some process that admits multiple maximizers. This situation arises frequently in problems concerning change-point estimation in regression settings. In these problems, the estimators are usually maximizers of processes that converge in the limit to two-sided, compound Poisson processes that have a complete interval of maximizers. See, for instance, \cite{koss} (Section 14.5.1, pages 271--277), \cite{lmm}, \cite{koso}, \cite{pons} and \cite{seseboot}. This issue has been noted before by several authors, such as \cite{fe04}.

The main goal of this paper is to derive a version of the argmax continuous mapping theorem specially taylored for situations like the one described in the previous paragraph. A distinctive feature of the argmax continuous mapping theorem in this setup is that it requires the weak convergence, not only of the objective functions, but also of some associated {\it pure jump processes}. Although this requirement has been overlooked by some authors in the past (we discuss these omissions in Section \ref{s4}), its necessity can be easily seen; see Section \ref{s4s1} for an example.


To illustrate the situations on which our results are applicable, we start with the following simple problem that arises in least squares change-point regression. Detailed accounts of this type of models can be found in \cite{koss} (Section 14.5.1, pages 271--277), \cite{lmm} and \cite{seseboot}. In its simplest form the model considers a random vector $X=(Y,Z)$ satisfying the following relation:
\begin{equation}\label{ec1ex1}
Y = \alpha_0 \mathbf{1}_{Z\leq \zeta_0} + \beta_0 \mathbf{1}_{Z> \zeta_0} + \epsilon,
\end{equation}
where $Z$ is a continuous random variable, $\alpha_0 \neq \beta_0 \in \mathbb{R}$, $\zeta_0 \in [c_1,c_2] \subset \mathbb{R}$ and $\epsilon$ is a continuous random variable, independent of $Z$ with zero expectation and finite variance $\sigma^2 > 0$. The parameter of interest is $\zeta_0$, the change-point. Given a random sample from this model, the {\it least squares estimator} $\hat{\theta}_n$ of $\theta_0 = (\zeta_0,\alpha_0, \beta_0)\in\Theta:=[c_1,c_2]\times\mathbb{R}^2$ is obtained by maximizing the criterion function

\begin{equation}\nonumber\function{M_n}{\theta} := - \frac{1}{n} \sum_{i=1}^n \left( Y_i - \alpha \mathbf{1}_{Z_i \leq \zeta} + \beta \mathbf{1}_{Z_i > \zeta} \right)^2,
\end{equation}
i.e., 
\begin{eqnarray}
\hat{\theta}_n := (\hat{\zeta}_n,\hat{\alpha}_n,\hat{\beta}_n) = \sargmax_{\theta \in \Theta} \left\{ M_n(\theta) \right\}, \label{ecfinal}
\end{eqnarray}
where sargmax denotes the maximizer with the smallest $\zeta$ value. This distinction is made as there is no unique maximizer for $\zeta$, in fact, for any $\alpha, \beta$, $M_n(\cdot,\alpha,\beta)$ is constant on every interval $[Z_{(j)},Z_{(j+1)})$, where $Z_{(j)}$ stands for the $j$-th order statistic. It can be shown, see either \cite{koss} (Section 14.5.1, pages 271--277) or  \cite{seseboot}, that $n(\hat{\zeta}_n-\zeta_0)$ converges in distribution to the smallest maximizer a two-sided, compound Poisson process. The convergence results in this paper, Theorems \ref{t1} and \ref{t2}, can, in particular, be applied to derive the asymptotic distribution of this estimator (see Section \ref{s4s4}).

Our results will be applicable to M-estimation problems for which the objective function takes arguments in some compact rectangle $K\subset\mathbb{R}^{d}$, $d \ge 1$. We focus on functions belonging to the Skorohod space $\mathcal{D}_K$ as defined in \cite{neu}. The elements of $\mathcal{D}_K$ are functions with finite ``quadrant limits'' (generalized one-sided limits) and are ``continuous from above'' (generalization of right-continuity) at each point in $K$. In Section \ref{s2} we describe the Skorohod space $\mathcal{D}_K$ in details and state some fundamental properties of the sargmax functional. Some of the results developed in this connection can also be of independent interest. In Section \ref{s3} we prove a version of the continuous mapping theorem for the sargmax functional for elements of $\mathcal{D}_K$ which are c\'adl\'ag in the first component and jointly continuous on the last $d-1$. In Section \ref{s4s1} we describe an example that illustrates the necessity of the convergence of the associated pure jump processes in the results of Section \ref{s3}. Finally, in Section \ref{s4} we apply the theorems of Section \ref{s3} to the change-point regression problem described above and to the estimation of a change-point in time and in a covariate in the Cox-proportional hazards model.

\section{The Skorohod space $\mathcal{D}_K$}\label{s2}
\subsection{Definition and basic properties}\label{s2s1}
We start by recalling the Skorohod space as discussed in \cite{neu}. To simplify notation, we write the coordinates of any vector in $\mathbb{R}^d$ with upper indices. We consider a compact rectangle $K = [a,b] = [a^1,b^1] \times \cdots \times [a^d,b^d]$ for some $a < b\in\mathbb{R}^d$ with the inequality holding componentwise. For any space $\mathbb{R}^m$ we will write $|\cdot|$ for the Euclidian norm (although the $\mathbb{L}^\infty$-norm is used in \cite{neu}, the results in there hold if one uses the Euclidian norm instead). For $k\in\{1,\ldots,d\}$, $t\in [a^k,b^k]$ and $s\in\{a^k,b^k\}$ we write:
\begin{eqnarray*}
I_k (s,t) &:=& \left\{\begin{array}{cl}
[a^k, t) & \textrm{ if } s=a^k,\\
(t,b^k] & \textrm{ if } s=b^k.
\end{array}\right. \label{ec1i}\\
J_k (s,t) &:=& \left\{\begin{array}{cl}
[a^k,t) & \textrm{ if } s=a^k \textrm{ and } t< b^k,\\
\left[a^k,b^k\right] & \textrm{ if } s=a^k \textrm{ and } t = b^k,\\
\emptyset & \textrm{ if } s=b^k \textrm{ and } t = b^k,\\
\left[t,b^k\right] & \textrm{ if } s=b^k \textrm{ and } t < b^k.
\end{array}\right.\label{ec2i}
\end{eqnarray*}
and for any $\displaystyle \rho\in\mathcal{V}:=\prod_{k=1}^d \{a^k,b^k\},x=(x^1,\ldots,x^d)\in\mathbb{R}^d$,
\begin{eqnarray}
Q(\rho,x) &:=& \prod_{k=1}^d I_k(\rho^k,x^k)\nonumber,\\
\tilde{Q}(\rho,x) &:=& \prod_{k=1}^d J_k(\rho^k,x^k)\nonumber.
\end{eqnarray}

\noindent {\bf Remark:} Some properties of the sets $\tilde{Q}(\rho,x)$ are:
\begin{enumerate}[(a)]
\item $\tilde{Q}(\rho,x) \cap \tilde{Q}(\gamma,x) = \emptyset$ for every $\gamma\neq \rho\in\mathcal{V}$ and every $x\in K$.
\item $\displaystyle K = \bigcup_{\rho\in\mathcal{V}}\tilde{Q}(\rho,x)$ for every $x\in K$.
\end{enumerate}
Hence, $\left\{\tilde{Q}(\rho,x)\right\}_{\rho\in\mathcal{V}}$ forms a partition of $K$. We are now in a position to define the so-called quadrant limits, the concept of continuity from above and the Skorohod space.

\begin{defn}[Quadrant Limits and Continuity from Above]\label{d1}
$\\$
Consider a function $f:\mathbb{R}^d\rightarrow \mathbb{R}$, $\rho\in\mathcal{V}$ and $x\in K$. We say that a number $l$ is the $\rho$-limit of $f$ at $x$ if for every sequence $\{x_n\}_{n=1}^\infty\subset Q(\rho,x)$ satisfying $x_n\rightarrow x$ we have $f(x_n)\rightarrow l$. In this case we write $l = f(x+ 0_\rho)$. When $\rho = b$ we may write $f(x+0_+):=f(x+0_b)$. With this notation, $f$ is said to be continuous from above at $x$ if $f(x+0_+) = f(x)$.
\end{defn}

\begin{defn}[The Skorohod Space]\label{d2}
$\\$
We define the Skorohod space $\mathcal{D}_K$ as the collection of all functions $f:K\rightarrow\mathbb{R}$ which have all $\rho$-limits and are continuous from above at every $x\in K$.
\end{defn}

\noindent {\bf Remark:} It is easily seen that if $f\in\mathcal{D}_K$, $\rho\in\mathcal{V}$, $x\in K$ and $\{x_n\}_{n=1}^\infty\subset \tilde{Q}(\rho,x)$ is a sequence with $x_n\rightarrow x$, then $f(x_{n})\rightarrow f(x+0_\rho)$. This follows from the continuity from above as $Q(\rho,x)\cap Q(b,\xi)\neq \emptyset$ for every $\xi\in\tilde{Q}(\rho,x)$.


Before stating some of the most important properties of $\mathcal{D}_K$ we will introduce some further notation. Consider the partitions $\mathcal{T}_j = \{ a^j = t_{j,0} < t_{j,1} < \ldots < t_{j,{r_j}} = b^j\}$ for $j=1,\ldots,d$. We define the rectangular partition $\mathcal{R}(\mathcal{T}_1,\ldots,\mathcal{T}_d)$ determined by $\mathcal{T}_1,\ldots,\mathcal{T}_d$ as the collection of all rectangles of the form
\[ R = \prod_{k=1}^d \left[t_{k,j_k-1},t_{k,j_k}\rangle\right., j_k\in\{1,\ldots,r_k\},\ k=1,\ldots,d, \]
where $\rangle$ stands for ``)'' or ``]'' if $t_{k,j_k}<b^{k}$ or $t_{k,j_k}=b^{k}$, respectively. With the aid of this notation, we can now state two important lemmas.

\begin{lemma}\label{l1}
$\\$
Let $f\in\mathcal{D}_K$. Then, for every $\epsilon > 0$ there is $\delta>0$ and partitions $\mathcal{T}_j$ of $[a^j,b^j]$, $j=1,\ldots,d$, such that for any $R\in\mathcal{R}(\mathcal{T}_1,\ldots,\mathcal{T}_d)$ and any $\theta,\vartheta\in R$ with $|\theta-\vartheta|<\delta$ the inequality $|f(\theta)-f(\vartheta)|<\epsilon$ holds. Furthermore, we can take the partitions in such a way that $\displaystyle \sup_{\theta,\vartheta\in R}\{|\theta - \vartheta|\} < \delta$ for every $R\in\mathcal{R}(\mathcal{T}_1,\ldots,\mathcal{T}_d)$.
\end{lemma}

\begin{lemma}\label{l2}
$\\$
Every function in $\mathcal{D}_K$ is bounded on $K$.
\end{lemma}

Lemmas \ref{l1} and \ref{l2} are, respectively, Lemma 1.5 and Corollary 1.6 in \cite{neu}. Their proofs can be found there.

Let $K_1 = [a^1,b^1]$ and $K_2 = [a^2,b^2]\times \cdots \times [a^d,b^d]$, so $K = K_1 \times K_2$. We will be dealing with functions which are c\'adl\'ag on the first coordinate and continuous on the remaining $d-1$. For this purpose we will turn our attention to the space $\widetilde{\mathcal{D}}_K\subset\mathcal{D}_K$ of all functions $f\in\mathcal{D}_K$ such that $f(t,\cdot):K_2\rightarrow\mathbb{R}$ is continuous $\forall$ $t\in K_1$ and $f(\cdot,\xi):K_1\rightarrow\mathbb{R}$ is c\'adl\'ag $\forall$ $\xi\in K_2$.

\noindent {\bf Remark:} It is worth noting that \textit{all elements in} $\mathcal{D}_K$ {\it are componentwise c\'adl\'ag}, so it is really the continuity in the last $d-1$ coordinates what makes $\widetilde{\mathcal{D}}_K$ a proper subspace of $\mathcal{D}_K$.

\begin{lemma}\label{l5}
$\\$
Let $f\in\widetilde{\mathcal{D}}_K$ and $\epsilon > 0$. Then, there is $\delta > 0$ such that
\[ \sup_{\begin{subarray}{c} |\xi - \eta| < \delta \\ \xi,\eta\in K_2\end{subarray}}\{ |f(t,\xi) - f(t,\eta)|\} \leq \epsilon\ \ \ \forall\ t\in K_1.\]
\end{lemma}
\begin{proof}
From Lemma \ref{l1} we can find $\delta_0>0$ and partitions $\mathcal{T}_j$ of $[a^j,b^j]$, $j=1,\ldots,d$ such that the conclusions of the lemma hold true with $\epsilon$ replaced by $\frac{\epsilon}{3}$. We take the partitions in such a way that whenever $\theta$ and $\vartheta$ belong to the same rectangle, the distance between them is less than $\delta_0$. Let $s\in\mathcal{T}_1$. Since $K_2$ is compact and $f(s,\cdot)$ is continuous, we can find $\delta_s$ such that for any $\xi,\eta\in K_2$ with $|\xi-\eta|<\delta_s$ we get $|f(s,\xi) - f(s,\eta)|<\frac{\epsilon}{3}$. Let $\displaystyle \delta = \min_{s\in\mathcal{T}_1}\{\delta_s\}$ and pick $t\in K_1$ and $\xi,\eta\in K_2$ with $|\xi - \eta| < \delta$. Take the largest $s\in\mathcal{T}_1$ with $s\leq t$. Then, $|s-t|<\delta_0$ and hence
\[ |f(t,\eta) - f(t,\xi)| \leq |f(t,\xi) - f(s,\xi)| + |f(s,\eta) - f(s,\xi)| + |f(t,\eta) - f(s,\eta)| < \epsilon. \]
The proof is then finished by taking the supremum over $\xi$ and $\eta$ and noticing that the choice of $\delta$ was independent of $t$.
\end{proof}

\subsection{The Skorohod topology}\label{s2s2}
So far we have not yet defined a topology on $\mathcal{D}_K$, so we turn our attention to this issue now. We will start by defining the Skorohod metric as given in \cite{neu}. Then, we will define a second metric on $\widetilde{D}_K$ and show that it is equivalent to the corresponding restriction of the Skorohod metric. This second metric will be more natural for the structure of $\widetilde{D}_K$ and will prove useful in the proof of the continuous mapping theorem for the smallest argmax functional. In order to define both of these metrics and state some of their properties, we will need some additional notation.

Consider a closed interval $I\subset\mathbb{R}$ and the class $\Lambda_I$ of all functions $\lambda:I\rightarrow I$ which are surjective (onto) and strictly monotone increasing. Define the function $\interleave\cdot\interleave_I:\Lambda_I\rightarrow\mathbb{R}$ by the formula $\displaystyle \interleave\lambda\interleave_I = \sup_{s\neq t}\left\{\left| \function{\log}{\frac{\lambda(t)-\lambda(s)}{t-s}}\right|\right\}$. We write $\Lambda_K := \Lambda_{[a^1,b^1]}\times\cdots\times\Lambda_{[a^d,b^d]}$ and for $\lambda:=(\lambda_1,\ldots,\lambda_d)\in\Lambda_K$, $\displaystyle \interleave \lambda \interleave_K := \max_{1\leq k \leq d} \{\interleave\lambda_k\interleave_{[a^k,b^k]}\}$. In a similar fashion, we define $\Lambda_{K_2} := \Lambda_{[a^2,b^2]}\times\cdots\times\Lambda_{[a^d,b^d]}$ and for $\lambda\in\Lambda_{K_2}$, $\displaystyle \interleave \lambda \interleave_{K_2} := \max_{2\leq k \leq d} \{\interleave\lambda_k\interleave_{[a^k,b^k]}\}$. Note that for $(\lambda_1,\lambda)\in\Lambda_K = \Lambda_{K_1}\times\Lambda_{K_2}$ we have $\interleave (\lambda_1,\lambda) \interleave_{K} = \interleave \lambda_1 \interleave_{K_1}\lor\interleave \lambda \interleave_{K_2}$.   We will use the sup-norm notation also: for a function $f:A\rightarrow\mathbb{R}$ we write $\displaystyle \|f\|_A=\sup_{x\in A}\{|f(x)|\}$.

\begin{defn}[The Skorohod metric]
$\\$
We define the Skorohod metric $d_K:\mathcal{D}_K\times\mathcal{D}_K\rightarrow\mathbb{R}$ as follows:
{ \[ d_K(f,g) = \inf_{\lambda\in\Lambda_K}\left\{\interleave\lambda\interleave_K + \|f-g\circ\lambda\|_K\right\}. \]}
\end{defn}

With this definition we can now state the following fundamental result about the Skorohod space.

\begin{lemma}\label{l6}
$\\$
The Skorohod metric is a metric. If $\mathcal{D}_K$ is endowed with the topology defined by $d_K$, then it becomes a Polish space.
\end{lemma}

For a proof of the last result, we refer the reader to Section 2 in \cite{neu}. We now proceed to define another metric, $\widetilde{d}_K$, on $\mathcal{D}_K$ by the formula:
\begin{equation}\nonumber
\widetilde{d}_K(f,g) = \inf_{\lambda\in\Lambda_{[a^1,b^1]}}\left\{\interleave\lambda\interleave_{[a^1,b^1]} + \sup_{(t,\xi)\in K_1\times K_2} \{ |f(t,\xi)-g(\lambda(t),\xi)|\}\right\}.
\end{equation}
To properly describe the properties of $\widetilde{d}_K$ we need the ball notation for metric spaces: given a metric space $(\texttt{X},d)$, $r>0$ and $x\in\texttt{X}$ we write $B_{r}^d(x)$ for the open ball of radius $r$ and center at $x$ with respect to the metric $d$. Additionally, the following lemma will prove to be useful.

\begin{lemma}\label{l8}
$\\$
Let $I\subset\mathbb{R}$ be any compact interval. Then, for $\epsilon>0$ there is $\delta>0$ such that for any $\lambda\in\Lambda_I$ with $\interleave \lambda \interleave_I<\delta$ we also have $$\displaystyle \sup_{s\in I}\{|\lambda(s) - s|\} <\epsilon.$$
\end{lemma}
\begin{proof}
Assume that $I=[u,v]$. It suffices to choose $\delta < \frac{1}{4}\land\frac{\epsilon}{2|v-u|}$. To see this, observe that for any $\tau\in(0,\frac{1}{4})$, $\tau < 2\tau - 4\tau^2\leq \log(1+2\tau)$  and for any $\tau>-1$, $\log(1+\tau) \leq \tau$. It follows that for $\lambda\in\Lambda_I$ with $\interleave \lambda \interleave_I<\delta$ and any $s\in I$, $\log(1-2\delta)< -\delta \leq \log\frac{\lambda(s) -u}{s - u} \leq \delta < 2\delta - 4\delta^2\leq \log(1+2\delta)$ and thus, $|\lambda(s) -s|<2(s-u)\delta\leq 2|u-v|\delta$. In the previous inequalities we have made implicit use of the fact that $\lambda(u)=u$.
\end{proof}

The next lemma contains some of the most relevant properties of $\widetilde{d}_K$.

\begin{lemma}\label{l7}
$\\$
The following statements are true:
\begin{enumerate}[(i)]
\item $\widetilde{d}_K$ is a metric on $\mathcal{D}_K$.
\item $d_K (f,g) \leq\widetilde{d}_K(f,g) \leq \|f-g\|_K$ $\forall$ $f,g\in\mathcal{D}_K$.
\item If $f\in\widetilde{\mathcal{D}}_K$, then for every $r>0$ there is $\delta>0$ such that $B_{\delta}^{d_K}(f)\subset B_{r}^{\widetilde{d}_K}(f)$. Moreover, the metrics $d_K$ and $\widetilde{d}_K$ generate the same topology on $\widetilde{\mathcal{D}}_K$.
\item If $f$ is continuous, then for every $r>0$ there is $\delta>0$ such that $B_{\delta}^{
\widetilde{d}_K}(f)\subset B_{r}^{\|\cdot\|_K}(f)$. Moreover, the metrics $d_K$ and $\widetilde{d}_K$ and $\|\cdot\|_K$ generate the same topology on the space of continuous functions on $K$.
\item $(\widetilde{\mathcal{D}}_K,\widetilde{d}_K)$ is a Polish space.
\end{enumerate}
\end{lemma}
\begin{proof}
It is straightforward to see that (ii) holds. The proof of (i) follows along the lines of the proof of the analogous results for the classical Skorohod metric (see Chapter 3 of \cite{bi}). For the sake of brevity we omit these arguments. For (iii) we use Lemma \ref{l5}. Let $f\in\widetilde{\mathcal{D}}_K$, $r>0$ and take $\delta_1>0$ such that the conclusions of Lemma \ref{l5} hold with $\frac{r}{3}$ replacing $\epsilon$. Also, consider $\delta_2>0$ such that $\interleave \lambda \interleave_{K_2} <\delta_2$ implies $\displaystyle \sup_{\xi\in K_2}\{|\lambda(\xi)-\xi|\}<\delta_1$ (whose existence is a consequence of Lemma \ref{l8} applied to each of the intervals $[a^2,b^2],\ldots,[a^d,b^d]$). Let $\delta = \delta_2\land\frac{r}{3}$ and take $g\in B_{\delta}^{d_K}(f)$. Find $(\lambda_1,\lambda)\in\Lambda_K = \Lambda_{K_1}\times\Lambda_{K_2}$ such that $\interleave (\lambda_1,\lambda)\interleave_K<\delta$ and $\|g-f\circ (\lambda_1,\lambda)\|_K<\frac{r}{3}$. Then, for any $(t,\xi)\in K_1\times K_2$ we have:
\begin{eqnarray*}
|g(t,\xi) - f(\lambda_1(t),\xi)| &\leq& |g(t,\xi) - f(\lambda_1(t),\lambda(\xi))|+ |f(\lambda_1(t),\lambda(\xi)) - f(\lambda_1(t),\xi)|\\
 &<& \frac{r}{3} + \frac{r}{3},
\end{eqnarray*}
where the second term in the sum of the right-hand side of the first inequality in the preceding display is less than $\frac{r}{3}$ because of Lemma \ref{l5} since $\interleave \lambda \interleave_{K_2}<\delta_2$. Taking supremum over $(t,\xi)\in K$ and considering that $\interleave \lambda_1 \interleave_{K_1}<\frac{r}{3}$ we get that $\widetilde{d}_K(f,g) < r$. Thus, $B_{\delta}^{d_K}(f)\subset B_{r}^{\widetilde{d}_K}(f)$. Taking (ii) into account we can conclude that $\widetilde{d}_K$ and $d_K$ are equivalent metrics on $\widetilde{\mathcal{D}}_K$.

We now turn out attention to (iv). Let $r >0$. Then, there is $\delta_1>0$ such that $|f(x) - f(y)| < \frac{r}{2}$ whenever $|x-y|<\delta_1$. Also, there is $\delta_2>0$ such that $\interleave \lambda \interleave_{K_1} <\delta_2$ implies $\displaystyle \sup_{t\in K_1}\{|\lambda(t) - t|\}<\delta_1$. Let $\delta = \delta_2\land\frac{r}{2}$ and let $g\in\mathcal{D}_K$ with $\widetilde{d}_K(f,g)<\delta$ and $\lambda\in\Lambda_{K_1}$ such that $\interleave \lambda \interleave_{K_1} + \|g(\cdot,\cdot) - f(\lambda(\cdot),\cdot)\|_{K_1\times K_2} < \delta$. Then, for any $(t,\xi)\in K_1\times K_2$ we have
\[ |f(t,\xi) - g(t,\xi)| \leq |f(t,\xi) - f(\lambda(t),\xi)| + |f(\lambda(t),\xi) - g(t,\xi)| < r. \]
Thus, $B_{\delta}^{\widetilde{d}_K}(f)\subset B_{r}^{\|\cdot\|_K}(f)$.

To prove (v) it suffices to show that $\widetilde{\mathcal{D}}_K$ is a closed subset of $\mathcal{D}_K$, as the latter space is known to be Polish (see \cite{neu}). Let $(f_n)_{n=1}^\infty$ be a sequence in $\widetilde{\mathcal{D}}_K$ such that $f_n\stackrel{d_K}{\longrightarrow} f$ for some $f\in\mathcal{D}_K$. We will show that $f(t,\cdot)$ is continuous for every $t$ and that will imply that $f\in\widetilde{\mathcal{D}}_K$ since $f$ is automatically componentwise c\'adl\'ag. Let $(t,\xi)\in K_1\times K_2 = K$ and $\epsilon>0$. Consider $n\in\mathbb{N}$ large enough so that $d_K(f,f_n) < \frac{\epsilon}{3}$ and take $\delta_1>0$ such that the conclusions of Lemma \ref{l5} hold true for $f_n$ and $\frac{\epsilon}{3}$. Let $(\lambda_{n,1},\lambda_n)\in\Lambda_{K_1}\times\Lambda_{K_2}$ such that $\interleave (\lambda_{n,1},\lambda_n)\interleave_K  + \|f - f_n\circ(\lambda_{n,1},\lambda_n)\|_K < \frac{\epsilon}{3}$. Since $\lambda_n$ is continuous, there is $\delta>0$ such that $|\xi - \eta|<\delta$ implies $|\lambda_n(\xi) - \lambda_n(\eta)|<\delta_1$. It follows that $|f_n(\lambda_{n,1}(t),\lambda_n(\xi)) - f_n(\lambda_{n,1}(t),\lambda_n(\eta))|<\frac{\epsilon}{3}$ whenever $|\xi - \eta|<\delta$. Hence,
\begin{eqnarray*}
|f(t,\xi) - f(t,\eta)| &\leq& |f(t,\xi) - f_n(\lambda_{n,1}(t),\lambda_n(\xi))|+|f(t,\eta) - f_n(\lambda_{n,1}(t),\lambda_n(\eta))|\\
 & & + |f_n(\lambda_{n,1}(t),\lambda_n(\xi)) - f_n(\lambda_{n,1}(t),\lambda_n(\eta))|\\
 &<& \epsilon,\ \ \ \forall\ \xi,\eta\in K_2 \textrm{ such that } |\xi-\eta|<\delta.
\end{eqnarray*}
It follows that $f(t,\cdot)$ is continuous for every $t\in K_1$. Hence, $f\in\widetilde{\mathcal{D}}_K$ and $\widetilde{\mathcal{D}}_K$ is closed.
\end{proof}

\noindent {\bf Remark:} Observe that the previous lemma implies that for a convergent sequence in $\mathcal{D}_K$ with a limit in $\widetilde{\mathcal{D}}_K$ convergence in the $\widetilde{d}_K$ and $d_K$ metrics are equivalent. When the limit is continuous, convergence in any of these metrics is equivalent to convergence in the sup-norm topology.

\subsection{The sargmax functional on $\mathcal{D}_K$}\label{s2s3}
We now turn our attention to the smallest argmax functional on $\mathcal{D}_K$.

\begin{defn}[The sargmax Functional]\label{sargmax}
A function $f\in\mathcal{D}_K$ is said to have a maximizer at a point $x\in K$ if any of the quadrant-limits of $x$ equals $\displaystyle \sup_{\xi\in K}\{ f(\xi)\}$. For any $f\in\mathcal{D}_K$ we can define the {\it smallest argmax} of $f$ over the compact rectangle $K$, denoted by $\displaystyle \sargmax_{x\in K}\{f(x)\}$, as the unique element $x=(x^1,\ldots,x^d)\in K$ satisfying the following properties:
\begin{enumerate}[(i)]
\item $x$ is a maximizer of $f$ over $K$,
\item if $\xi=(\xi^1,\ldots,\xi^d)$ is any other maximizer, then $x^1\leq \xi^1$,
\item if $\xi$ is any maximizer satisfying $x^j = \xi^j$ $\forall$ $j=1,\ldots,k$ for some $k\in\{1,\ldots,d-1\}$, then $x^{k+1} \leq \xi^{k+1}$.
\end{enumerate}
We say that $x$ is the largest maximizer of $f$, denoted by $\displaystyle \largmax_{\xi\in K}\{f(\xi)\}$, if it is a maximizer that satisfies $(ii)$ and $(iii)$ above with the inequalities reversed.
\end{defn}

The first question that one might ask is whether or not the sargmax is well defined for all functions in the Skorohod space. Before attempting to give an answer, we will use our notation to clarify the concept of a maximizer: a point $x\in K$ is a maximizer of $f\in\mathcal{D}_K$ if
\begin{equation}\nonumber
\max_{\rho\in\mathcal{V}}\{f(x+0_\rho)\} = \sup_{\xi\in K}\{f(\xi)\}.
\end{equation}
We can now prove a result concerning the set of maximizers of a function in $\mathcal{D}_K$.

\begin{lemma}\label{l3}
$\\$
The set of maximizers of any function in $\mathcal{D}_K$ is compact.
\end{lemma}
\begin{proof}
Let $f\in\mathcal{D}_K$. Since the set of maximizers of $f$ is a subset of the compact rectangle $K$, it suffices to show that any convergent sequence of maximizers converges to a maximizer. Let $(x_n)_{n=1}^\infty$ be a sequence of maximizers with limit $x$. For each $x_n$ we can find $\xi_n$ with $|x_n - \xi_n|<\frac{1}{n}$ and such that $|f(\xi_n) - \max_{\rho\in\mathcal{V}}\{f(x_n+0_\rho)\}|<1/n$. Then we have that $\xi_n\rightarrow x$ and $|f(\xi_n) - \sup_{\xi\in K}\{f(\xi)\}|<1/n$ $\forall$ $n\in\mathbb{N}$. Since $K$ is the disjoint union of $\{\tilde{Q}(\rho,x)\}_{\rho\in\mathcal{V}}$, it follows that there is $\rho_*\in\mathcal{V}$ and a subsequence $(\xi_{n_k})_{k=1}^\infty$ such that $\xi_{n_k}\in\tilde{Q}(\rho_*,x)$ $\forall$ $k\in\mathbb{N}$. Therefore, the remark stated right after the definition of the Skorohod space implies that $f(\xi_{n_k})\rightarrow f(x+0_{\rho_*})$ and, consequently, $\displaystyle f(x+0_{\rho_*}) = \sup_{\xi\in K}\{f(\xi)\}$.
\end{proof}

The previous lemma can be used to show that the sargmax functional is well defined on $\mathcal{D}_K$.

\begin{lemma}\label{l4}
$\\$
For each $f\in\mathcal{D}_K$ there is a unique element in $x\in K$ such that $\displaystyle x = \sargmax_{\xi\in K}\{f(\xi)\}$.
\end{lemma}
\begin{proof}
Let $f\in\mathcal{D}_K$. Since the set of maximizers of $f$ is compact, if we can show that it is nonempty then the compactness will imply that there is a unique element $x\in{K}$ satisfying properties (i), (ii) and (iii) of Definition \ref{sargmax}. Hence, it suffices to show that $f$ has at least one maximizer. For this purpose, for each $n\in\mathbb{N}$ choose $x_n$ such that $\displaystyle \sup_{\xi\in K}\{f(\xi)\}< f(x_n) + \frac{1}{n}$. Since $K$ is compact, there is $x\in K$ and a subsequence $(x_{n_k})_{k=1}^\infty$ such that $x_{n_k}\rightarrow x$. Just as in the proof of the previous lemma, we can find $\rho_*\in\mathcal{V}$ and a further subsequence $(x_{n_{k_s}})_{s=1}^\infty$ such that $x_{n_{k_s}}\in\tilde{Q}(\rho_*,x)$ $\forall$ $s\in\mathbb{N}$. It follows that $f(x_{n_{k_s}})\rightarrow f(x+0_{\rho_*})$ and hence $\displaystyle \sup_{\xi\in K}\{f(\xi)\} = f(x+0_{\rho_*})$. Therefore, the set of maximizers is nonempty and the sargmax is well defined.
\end{proof}

We finish this section with a continuity theorem for the sargmax functional on continuous functions.

\begin{lemma}\label{l9}
$\\$
Let $W\in\mathcal{D}_K$ be a continuous function which has a unique maximizer $x^* \in K$. Then, the smallest argmax functional is continuous at $W$ (with respect to $d_K$, $\widetilde{d}_K$ and the sup-norm metric).
\end{lemma}
\begin{proof}
Let $(W_n)_{n=1}^\infty$ be a sequence converging to $W$ in the Skorohod topology. Let  $\epsilon>0$ be given and $G$ be the open ball of radius $\epsilon$ around $x^*$ and let $\delta := \left(W(x^*) - \sup_{x\in K\setminus G}\left\{W(x)\right\}\right)/2 > 0$. By Lemma \ref{l7} we have $\left\|W_n - W\right\|_K<\delta$ for all large $n$ ($d_K$, $\widetilde{d}_K$ and $\|\cdot\|_K$ generate the same local topology on $W$). Then
\[ W(x^*)=2\delta + \sup_{x\in K\setminus G}\left\{W(x)\right\}> \delta + \sup_{x\in K\setminus G}\left\{W_n(x)\right\}.\]
But $\left\|W_n - W\right\|_K<\delta$ also implies that $\displaystyle \sup_{x\in K}\{W_n(x)\} > W(x^*)-\delta$. The combination of these two facts shows that if $\left\|W_n - W\right\|_K<\delta$, then any maximizer of $W_n$ must belong to $G$. Thus, $|\sargmax_{x \in K} \{W_n(x)\} - x^*|<\epsilon$ for $n$ large enough.
\end{proof}

\section{A continuous mapping theorem for the sargmax functional on functions with jumps}\label{s3}
Lemma \ref{l9} shows that the sargmax functional is continuous on continuous functions with unique maximizers. However, its raison d'\^etre is to fix a unique maximizer on a function having multiple maximizers. Thus, a continuous mapping theorem on functions with jumps and possibly multiple maximizers is desired. We will show a version of the continuous mapping theorem on a suitable subset of our space $\widetilde{\mathcal{D}}_K$.

To state and prove our version of the continuous mapping theorem for the sargmax functional, we need to introduce some notation. We start with the space $\mathcal{D}_K^0$ consisting of all functions $\psi:K_1\times K_2\rightarrow\mathbb{R}$ which can be expressed as:
\begin{eqnarray}
    \function{\psi}{t,\xi} = V_0 (\xi) \ind{a_{-1} \leq t < a_{1}} + \sum_{k=1}^\infty V_{k} (\xi)\ind{a_{k} \leq t < a_{k+1}} 
    + \sum_{k=1}^\infty  V_{-k}(\xi)\ind{a_{-k-1} \leq t < a_{-k}} \label{ec58}
\end{eqnarray}
where $\left(\ldots<a_{-k-1}<a_{-k}<\ldots< a_0 = 0<\ldots<a_{k}<a_{k+1}<\ldots\right)_{k\in\mathbb{N}}$ is a sequence of jumps and $\left(V_k\right)_{k\in\mathbb{Z}}$ is a collection of continuous functions. Note that $\mathcal{D}_K^0\subset\widetilde{\mathcal{D}}_K$. Observe that the representation in (\ref{ec58}) is not unique. However, knowledge of the function $\psi$ and of the jumps $(a_k)_{k\in\mathbb{Z}}$ completely determines the continuous functions $(V_k)_{k\in\mathbb{Z}}$.

Our theorem will require not only Skorohod convergence of the elements of $\mathcal{D}_K^0$, but also convergence of their associated {\it pure jump functions}. To define properly these jump functions, we introduce the space $\mathcal{S}$ all piecewise constant, c\'adl\'ag functions $\tilde{\psi}:\mathbb{R}\rightarrow\mathbb{R}$ such that $\tilde{\psi}(0)=0$; $\tilde{\psi}$ has jumps of size 1; and $\tilde{\psi}(-t)$ and $\tilde{\psi}(t)$ are nondecreasing on $(0,\infty)$. For any closed interval $I\subset\mathbb{R}$ we introduce the space $\mathcal{S}_I:=\{f|_I: f\in\mathcal{S}\}$. We endow the spaces $\mathcal{S}_I$ with the usual Skorohod topology $d_I$. Observe that the fact that all elements of $\mathcal{S}$ are c\'adl\'ag and have jumps of size one implies that {\it any function in $\mathcal{S}_I$ has a finite number of jumps on I}.

We associate with every $\psi\in\mathcal{D}_K^0$, expressed as in (\ref{ec58}), a pure jump function $\tilde \psi \in \mathcal{S}$ whose sequence of jumps is exactly the $a_k$'s, i.e.,
\begin{eqnarray}\label{eq:PureJumpProc}
\function{\tilde{\psi}}{t} &=& \sum_{k=1}^\infty \ind{a_{k} \leq t }+
\sum_{k=1}^\infty  \ind{a_{-k} > t}.
\end{eqnarray}

We will show that Skorohod-convergence of functions in $\mathcal{D}_K^0$ and Skorohod convergence of their associated pure jump functions implies convergence of the corresponding sargmax and largmax functionals.

The following convergence result is a generalization of both, Lemma 3.1 of \cite{lmm} and Lemma A.3 in \cite{seseboot}.

\begin{thm}\label{t1}
$\\$
Assume that $d\geq 2$ and let $\left(\psi_n,\tilde{\psi}_n\right)_{n=1}^\infty$, $(\psi_0,\tilde{\psi}_0)$ be functions in $\mathcal{D}_K^0 \times \mathcal{S}_{K_1}$ such that $\psi_n$ satisfies (\ref{ec58}) for the sequence of jumps of $\tilde{\psi}_n$ for any $n\geq 0$. Assume that $(\psi_n,\tilde{\psi}_n)\rightarrow (\psi_0,\tilde{\psi}_0)$ in $\mathcal{D}_K^0 \times \mathcal{S}_{K_1}$ (with the product topology). Suppose, in addition, that $\psi_0$ can be expressed as (\ref{ec58}) for the sequence of jumps $\left(\ldots<a_{-k-1}<a_{-k}< \ldots < a_0 = 0< \ldots< a_{k} \right.$ $\left. < a_{k+1}< \ldots\right)_{k\in\mathbb{N}}$ of $\tilde{\psi}_0$ and some continuous functions $(V_j)_{j\in\mathbb{Z}}$, each having a unique maximizer on $K_2$, with the property that for any finite subset $A\subset\mathbb{Z}$ there is only one $j\in A$ for which
\begin{equation}\label{ec24}
\max_{m\in A}\left\{\sup_{\xi\in K_2}\left\{V_m(\xi)\right\}\right\} = \sup_{\xi\in K_2}\left\{V_j(\xi)\right\}.
\end{equation}
Finally, assume that $\psi_0$ has no jumps at the extreme points of $K_1$.
Then,
\begin{enumerate}[(i)]
\item $\displaystyle \sargmax_{x\in K}\{\psi_n(x)\} \rightarrow \sargmax_{x\in K}\{\psi_0(x)\}$ as $n\rightarrow\infty$;
\item $\displaystyle \largmax_{x\in K}\{\psi_n(x)\} \rightarrow \largmax_{x\in K}\{\psi_0(x)\}$  as $n\rightarrow\infty$.
\end{enumerate}
The result is also true when $d=1$ under the same assumptions, but taking the sequence $(V_j)_{j\in\mathbb{Z}}$ to be a sequence of constants such that for any finite subset $A\subset\mathbb{Z}$ there is a unique $j\in A$ such that $\displaystyle \max_{m\in A}\{V_m\} = V_j$.
\end{thm}
\begin{proof} We focus on the case when $d>1$ as the one-dimensional case is just Lemma 3.1 of \cite{lmm}.  Without loss of generality, assume that $K_1 = [-C,C]$ for some $C>0$.

We can write $\psi_n$ in the form (\ref{ec58}) with $\left(\ldots<a_{n,-k-1}<a_{n,-k}<\right.$ $\\$ $\left.\ldots < a_{n,0} = 0 < \ldots\right.$ $\left.< a_{n,k}< a_{n,k+1}< \ldots\right)_{k \in \mathbb{N}}$ being the sequence of jumps of $\psi_n$ and $V_{n,j}$ being the continuous functions. Consequently, $\tilde \psi_n$, the pure jump function associated with $\psi_n$, can be expressed as (\ref{eq:PureJumpProc}) with jumps at $(a_{n,k})_{k \in \mathbb{Z}}$.

Let $N_r$ and $N_l$ be the number of jumps of $\tilde{\psi}_0$ in $[0,C]$ and $[-C,0)$ respectively. Let $\epsilon > 0$ be sufficiently small such that all the points of the form $a_j \pm \epsilon$ are continuity points of $\psi_0$, for $-N_l \leq j \leq N_r$. Since convergence in the Skorohod topology of $\tilde{\psi}_n$ to $\tilde{\psi}_0$ implies point-wise convergence for continuity points of $\tilde{\psi}_0$ (see page 121 of \cite{bi}), and all of them are integer-valued functions, we see that $\tilde{\psi}_n (a_j - \epsilon) = j-1$ and $\tilde{\psi}_n (a_j + \epsilon) = j$ for any $1\leq j\leq N_r$, and $\tilde{\psi}_n (C) = N_r$ for all sufficiently large $n$. Thus, for all but finitely many $n$'s we have that $\tilde{\psi}_n$ has exactly $N_r$ jumps between 0 and $C$ and that the location of the $j$-th jump to the right of 0 satisfies $|a_{n,j}-a_{j}|<\epsilon$. Since $\epsilon>0$ can be made arbitrarily small, we get that all the jumps $a_{n,j}$ converge to their corresponding $a_j$ for all $1\leq j\leq N_r$. The same happens to the left of zero: for all but finitely many $n$'s, $\tilde{\psi}_n$ has exactly $N_l$ jumps in $[-C,0)$ and the sequences of jumps $\left(a_{n,-j}\right)_{n=1}^\infty$, $1\leq j\leq N_l$, converge to the corresponding jumps $a_{-j}$.

Let $ V^* = \sup\left\{V_j (\xi): \xi\in K_2, -N_l\leq j \leq N_r\right\}$. Our assumptions on the $V_j$'s imply that this supremum is actually achieved at some unique vector $\xi^*\in K_2$ and that there is a unique ``flat stretch'' at which this supremum is attained (the last assertion follows form (\ref{ec24})).

Suppose, without loss of generality, that the maximum value is achieved in an interval of the form $[a_k,a_{k+1}\land C)$ for a unique $k\in\left\{1,\ldots, N_r\right\}$. Now, write $b_0 = 0$; $b_j=\frac{ a_j + C\land a_{j+1}}{2}$ for $1 \leq j \leq N_r$; and $b_j = \frac{a_j + (-C)\lor a_{j-1}}{2}$ for $-N_l \leq j \leq -1$. Note that the $b_j$'s (for any value of $\xi\in K_2$) are continuity points of both $\psi_0$ and $\tilde{\psi}_0$.

Let $\kappa = \min_{-N_l \leq j \leq N_r+1} (C \land a_j - (-C) \lor a_{j-1})$ be the length of the shortest stretch. Take $0 < \eta, \delta < \kappa/4$. Considering the convergence of the jumps of $\psi_n$ to those of $\psi_0$, there is  $N\in\mathbb{N}$ such that for any $n \geq N$, the following two statements hold:
\begin{enumerate}[(a)]
\item Consider $\rho > 0$ such that if $\interleave\lambda\interleave_{K_1}<\rho$, then \[\sup\left\{ |s - \lambda(s)| : s\in [-C,C]\right\}<\delta.\] The existence of such $\rho$ follows from Lemma \ref{l8}. By the convergence of $\psi_n$ to $\psi_0$ in the Skorohod topology, there exists $\lambda_n \in \Lambda_{K_1}$ such that $\interleave\lambda_n\interleave_{K_1} < \rho$ and \[\displaystyle \sup_{(t,\xi)\in K_1\times K_2} \left\{ |\psi_n(\lambda_n(t),\xi) - \psi_0(t,\xi)|\right\} < \eta.\]
\item For any $1\leq j\leq N_r$ (respectively, $j=0$, $-N_l\leq j \leq -1$), $b_j$ lies somewhere inside the interval $\left(a_{n,j} + \delta, C\land a_{n,j+1} - \delta\right)$ (respectively $\left(a_{n,-1} + \delta, \right.$ $\left. a_{n,1} - \delta\right)$, $\left((-C)\lor a_{n,j-1} + \delta, a_{n,j} - \delta \right)$). This follows from what was proven in the first two paragraphs of this proof.
\end{enumerate}
From (a) we see that $|\lambda_n (b_j) - b_j|< \delta$ for all $-N_l \leq j \leq N_r$. But (b) and the size of $\delta$ in turn imply that $b_j$ and $\lambda_n(b_j)$ belong to the same ``flat stretch'' of $\psi_n$ and thus $\psi_n(\lambda_n (b_j),\xi) = \psi_n(b_j,\xi) = V_{n,j}(\xi)$ for all $\xi\in K_2$ and all $-N_l \leq j \leq N_r$. Considering again (b) and the second inequality in (a), we conclude that $\left\|V_{n,j} - V_j\right\|_{K_2} < \eta$ for all $-N_l \leq j \leq N_r$ and all $n\geq N$. Hence, all the sequences $(V_{n,j})_{n=1}^\infty$ converge uniformly in $K_2$ to their corresponding $V_j$. Consequently:
\begin{eqnarray}
\max_{\substack{-N_l \leq j \leq N_r \\ j\neq k}}\left\{\sup_{\xi\in K_2} V_{n,j}(\xi)\right\}
&\longrightarrow&  \max_{\substack{-N_l \leq j \leq N_r \\ j\neq k}}\left\{\sup_{\xi\in K_2} V_{j}(\xi)\right\}, \nonumber\\
\max_{\xi\in K_2}\left\{ V_{n,k} (\xi) \right\}&\longrightarrow& \max_{\xi\in K_2}\left\{ V_{k}(\xi)\right\}  = V_k (\xi^*), \nonumber\\
\argmax_{\xi\in K_2}\left\{ V_{n,k} (h_1,h_2) \right\}&\longrightarrow& \argmax_{\xi\in K_2}\left\{ V_{k}(\xi)\right\}  = \xi^*, \nonumber\\
 \lsup_{n\rightarrow\infty} \max_{\substack{-N_l \leq j \leq N_r \\ j\neq k}}\left\{\sup_{\xi\in K_2} V_{n,j}(\xi)\right\} & < & \linf_{n\rightarrow\infty} \max_{\xi\in K_2}\left\{ V_{n,k}(\xi)\right\}. \nonumber
\end{eqnarray}
The above, together with (\ref{ec24}) and the fact that $a_{n,k}\rightarrow a_k$ and $a_{n,k+1}\rightarrow a_{k+1}$, imply that
\begin{itemize}
\item[] $\displaystyle \sargmax_{x\in K}\{\psi_n(x)\}\rightarrow (\xi^*,a_k) = \sargmax_{x\in K}\{\psi_0(x)\}$
\item[] $\displaystyle \largmax_{x\in K}\{\psi_n(x)\}\rightarrow (\xi^*,a_{k+1}) = \largmax_{x\in K}\{\psi_0(x)\}$
\end{itemize}
as $n\rightarrow\infty$.
\end{proof}

We now present a version of the previous result but for random elements in $\mathcal{D}_K^0$. To prove it, we will use Lemma 4.2 in \cite{prarao1969}. In the remaining of the paper we will use the symbol $\rightsquigarrow$ to represent weak convergence.

\begin{lemma}\label{l10}
$\\$
Consider the random vectors $\{W_{n\epsilon},W_n,W_\epsilon\}_{\epsilon\geq 0}^{n\in\mathbb{N}}$ and $W$. Suppose that the following conditions hold:
\begin{enumerate}[(i)]
\item $\displaystyle \lim_{\epsilon\rightarrow 0}\lsup_{n\rightarrow\infty}\p{W_{n\epsilon}\neq W_n} = 0$,
\item $\displaystyle \lim_{\epsilon\rightarrow 0}\p{W_\epsilon\neq W} = 0$,
\item $W_{n\epsilon}\rightsquigarrow W_\epsilon$ (as $n\rightarrow \infty$) for every $\epsilon>0$.
\end{enumerate}
Then, $W_n\rightsquigarrow W$.
\end{lemma}

In the next theorem we will be taking the sargmax and largmax functionals over rectangles that may not be compact. When this happens, we say that these functionals are {\it well defined} if there is an element in the corresponding rectangle satisfying conditions $(i)-(iii)$ defining the smallest and largest argmax functionals (see Definition \ref{sargmax}). If we are given a rectangle $\Theta\subset\mathbb{R}^d$ which can be written as the Cartesian product of possibly unbounded closed intervals, we will denote by $\mathcal{D}_\Theta$ the collection of functions $f:\Theta\rightarrow\mathbb{R}$ whose restrictions to all compact rectangles $K\subset\Theta$ belong to $\mathcal{D}_K$.

\begin{thm}\label{t2}
$\\$
Assume that $K=K_1\times K_2$ is a closed rectangle in $\mathbb{R}^d$ and that $0\in K_1^\circ$. Let $(\Omega,\mathcal{F},\mathbf{P})$ be a probability space and let $\left(\Psi_n,\Gamma_n\right)_{n=1}^\infty$, $(\Psi_0,\Gamma_0)$ be random elements taking values in $\mathcal{D}_K^0 \times \mathcal{S}_{K_1}$ such that $\Psi_n$ satisfies (\ref{ec58}) for the sequence of jumps of $\Gamma_n$ for any $n\geq 0$, almost surely. Moreover, suppose that, with probability one, we have that: $\Psi_0$ satisfies (\ref{ec24}); $\Gamma_0$ has no fixed time of discontinuity; the sargmax and largmax functionals over $K$ are finite for $\Psi_0$ (this assumption is essential as $K$ is not necessarily compact). If the following hold:
\begin{enumerate}[(i)]
\item For every compact subinterval $B_1\subset K_1$ and compact sub-rectangle $B:=B_1\times B_2\subset K$ we have $(\Psi_n,\Gamma_n)\rightsquigarrow(\Psi_0,\Gamma_0)$ on $\mathcal{D}_B\times\mathcal{D}_{B_1}$;
\item $\displaystyle \left(\sargmax_{\theta\in K}\{\Psi_n(\theta)\},\largmax_{\theta\in K}\{\Psi_n(\theta)\}\right) = O_\mathbf{P} (1)$;
\end{enumerate}
then we also have
\begin{equation}\nonumber
\left(\sargmax_{\theta\in K}\{\Psi_n(\theta)\},\largmax_{\theta\in K}\{\Psi_n(\theta)\}\right)\rightsquigarrow \left(\sargmax_{\theta\in K}\{\Psi_0(\theta)\},\largmax_{\theta\in K}\{\Psi_0(\theta)\}\right).
\end{equation}
\end{thm}
\begin{proof}
Consider $C>0$ and let
\begin{eqnarray*}
\phi_n&:=&\left(\sargmax_{\theta\in K}\{\Psi_n(\theta)\},\largmax_{\theta\in K}\{\Psi_n(\theta)\}\right)\\
\phi_{n,C}&:=&\left(\sargmax_{\theta\in [-C,C]^d\cap K}\{\Psi_n(\theta)\},\largmax_{\theta\in [-C,C]^d\cap K}\{\Psi_n(\theta)\}\right),
\end{eqnarray*}
for all $n\geq 0$. To prove the result, we will apply Theorem \ref{t1} and Lemma \ref{l10}. Using the notation of the latter, set $\epsilon = \frac{1}{C}$, $W_{n\epsilon} = \phi_{n,C}$ for $n\geq 1$, $W_\epsilon = \phi_{0,C}$, $W_n = \phi_n$ for $n\geq 1$ and $W=\phi_0$. From $(ii)$ we see that $\displaystyle \lim_{\epsilon\rightarrow 0} \lsup_{n\rightarrow\infty} \p{W_{n\epsilon} \neq W_n} = 0$. Our assumptions on $\Psi_0$ and $\Gamma_0$ imply that $\displaystyle \lim_{\epsilon\rightarrow 0} \p{ W_\epsilon \neq W} = 0$. Finally, Theorem \ref{t1} and an application of Skorohod's Representation Theorem (see either Theorem 1.8, page 102 in \cite{ek} or Theorems 1.10.3 and 1.10.4, pages 58 and 59 in \cite{vw}) show that $W_{n\epsilon} \rightsquigarrow W_{\epsilon}$ and hence, from Lemma \ref{l10}, we conclude that $\phi_n \rightsquigarrow \phi_0$. \end{proof}

\section{On the necessity of the convergence of the associated pure jump processes}\label{s4s1}
Condition (i) in Theorem \ref{t2} involves the joint convergence of the processes whose maximizers are being considered and their associated pure jump processes. One may ask whether or not this condition is actually necessary for the weak convergence of the corresponding smallest maximizers. A simple counterexample shows that such a condition is indeed essential to guarantee the desired weak convergence under the assumptions of Theorem \ref{t2}.

Let $\Psi$ be a two-sided, right-continuous Poisson process and $T_{\pm1}:=\pm\inf\{t>0: \Psi(\pm t)>0\}$. Consider the following $\mathcal{D}_\mathbb{R}$-valued random elements: $\Psi_0:= -\Psi$ and $\Psi_n = \Psi_0 + \frac{1}{n}\ind{\left[\frac{1}{2}T_{-1},\frac{1}{2}T_1\right)}$. Then, $\Psi_n\rightsquigarrow\Psi$ in $\mathcal{D}_I$ for every compact interval $I$ (in fact, the weak convergence holds in $\mathcal{D}_\mathbb{R}$ with the corresponding Skorohod topology). However,
\[\displaystyle \left(\sargmax_\mathbb{R}\{\Psi_n\},\largmax_\mathbb{R}\{\Psi_n\}\right) = \frac{1}{2}\left(\sargmax_\mathbb{R}\{\Psi_0\},\largmax_\mathbb{R}\{\Psi_0\}\right),\]
for all $n\in\mathbb{N}$. It is easily seen that all the conditions of Theorem 3.2 hold, with the exception of (i). Hence, the weak convergence of the processes $\Psi_n$ alone is not enough to guarantee weak convergence of the corresponding maximizers.

\section{Applications}\label{s4}
\subsection{Stochastic design change-point regression}\label{s4s4}
We start by analyzing the example of the least squares change-point estimator given by (\ref{ecfinal}) in the Introduction. Assume that we are given an i.i.d. sequence of random vectors $\left\{X_n=(Y_n,Z_n)\right\}_{n=1}^\infty$ defined on a probability space $\left(\Omega,\mathcal{A},\mathbf{P}\right)$ having a common distribution $\mathbb{P}$ satisfying (\ref{ec1ex1}) for some parameter $\theta_0 :=(\zeta_0,\alpha_0, \beta_0) \in \Theta := [c_1,c_2]\times\mathbb{R}^2$. Suppose that $Z$ has a uniformly bounded, strictly positive density $f$ (with respect to the Lebesgue measure) on $[c_1,c_2]$ such that $\inf_{|z - \zeta_0| \le \eta} f(z) > \kappa > 0$ for some $\eta >0$ and that $\mathbb{P}(Z<c_1)\land \mathbb{P}(Z>c_2) > 0$. For $\theta = (\zeta,\alpha,\beta) \in \Theta$, $x=(y,z) \in \mathbb{R}^2$ write
\begin{equation}\nonumber
\function{m_\theta}{x} := -\left(y - \alpha \mathbf{1}_{z\leq \zeta} - \beta \mathbf{1}_{z> \zeta}\right)^2,
\end{equation}
and $\mathbb{P}_n$ for the empirical measure defined by $X_1,\ldots,X_n$. Note that $\function{M_n}{\theta} := - \mathbb{P}_n [m_\theta]$ and recall the definition of $
\hat{\theta}_n$.

The asymptotic properties of this estimator are well-known and have been deduced by several authors. They are available, for instance, in \cite{koss} or \cite{seseboot}. It follows from Proposition 3.2 in \cite{seseboot} that $\sqrt{n}(\hat{\alpha}_n - \alpha_0) = \function{O_{\mathbf{P}}}{1}$, $\sqrt{n}(\hat{\beta}_n - \beta_0) = \function{O_{\mathbf{P}}}{1}$ and $n (\hat{\zeta}_n - \zeta_0) = \function{O_{\mathbf{P}}}{1}$.

For $h = (h_1,h_2,h_3) \in \mathbb{R}^3$, let $\vartheta_{n,h} := \theta_0 + \left(\frac{h_1}{n},\frac{h_2}{\sqrt{n}},\frac{h_3}{\sqrt{n}}\right)$ and
\begin{eqnarray}
\hat{E}_n(h) := n\mathbb{P}_n \left[m_{\vartheta_{n,h}} - m_{\theta_0} \right] \nonumber.
\end{eqnarray}
A consequence of the rate of convergence result in \cite{seseboot} is that with probability tending to one, we have
$$\hat{h}_n := \sargmax_{h \in \mathbb{R}^3} \hat E_n(h) = \left(n(\hat{\zeta}_n - \zeta_0),\sqrt{n}(\hat{\alpha}_n - \alpha_0), \sqrt{n}(\hat{\beta}_n - \beta_0) \right).$$
Write $\hat{J}_n$ for the pure jump process associated with $\hat{E}_n$. It is shown in Lemma 3.3 of \cite{sesetec} that
\begin{enumerate}[(a)]
\item $(\hat{E}_n,\hat{J}_n)\rightsquigarrow (E^*,J^*)$ in $\mathcal{D}_K\times \mathcal{S}_I$,
\end{enumerate}
on every compact rectangle $K=I\times A\times B\subset\mathbb{R}^3$ for some process $E^*\in\mathcal{D}_{\mathbb{R}^3}$ with an associated pure jump process $J^*$. Then, an application of Theorem \ref{t2} shows that
$$\displaystyle \hat{h}_n = \left(n(\hat{\zeta}_n - \zeta_0),\sqrt{n}(\hat{\alpha}_n - \alpha_0), \sqrt{n}(\hat{\beta}_n - \beta_0) \right) \rightsquigarrow \sargmax_{h\in\mathbb{R}^3}\{E^*(h)\}.$$
It must be noted that the results in \cite{seseboot} are stated in terms of a triangular array of random vectors that satisfy some regularity conditions. Even in such generality, Proposition 3.3 in \cite{seseboot} can be derived from Theorem \ref{t2}.

We would like to point out that the derivation of the asymptotic distribution of this estimator can also be found in \cite{koss}. The arguments there can be modified to obtain the result from an application of Theorem \ref{t2}.

\subsection{Estimation in a Cox regression model with a change-point in time}\label{s4s2}
Define $\Theta:=(0,1)\times\mathbb{R}^{p+2q}$ for given $p,q\in\mathbb{N}$. For $\theta=(\tau,\xi)=(\tau,\alpha,\beta,\gamma)\in\Theta=(0,1)\times \mathbb{R}^p\times \mathbb{R}^q\times \mathbb{R}^q$ consider a survival time $T^0$, a censoring time $C$ and covariate c\'agl\'ad (left-continuous with right-hand side limits) $\mathbb{R}^{p+q}$-valued process $Z=(Z_1,Z_2)$ where the sample paths of $Z_1$ and $Z_2$ live in $\mathbb{R}^p$ and $\mathbb{R}^q$, respectively. Assume that $C$ and $Z$ have laws $G$ and $H$, respectively. Note that $G$ is a distribution on the nonnegative real line and $H$ a probability measure on the space of left continuous processes with right-hand side limits. In our Cox model with a change-point in time we make the additional assumption that, conditionally on $Z$, the hazard function of the survival time is given by:
\begin{eqnarray}
\lambda(t|Z) &:=& \lim_{\Delta t\downarrow 0}\frac{\p{t\leq T^0 < t+\Delta t|T^0\geq t;\ Z(s),\ 0\leq s\leq t}}{\Delta t}\nonumber\\
 &=& \lambda(t)e^{\alpha\cdot Z_1(t) + (\beta+\gamma\ind{t>\tau})\cdot Z_2(t)} \nonumber
\end{eqnarray}
where $\lambda$ is the {\it baseline hazard function} and $\cdot$ denotes the standard inner product on Euclidian spaces. We write $\mathbb{P}_{\theta,\lambda,G,H}$ for the law of $(T^0,C,Z)$. We would like to point out that we assume that $G$ and the finite dimensional distributions of $Z$ are all {\it continuous}.

Suppose that there is a random sample \[(T^0_1,C_1,Z_{1,1},Z_{2,1}),\ldots,(T^0_n,C_n,Z_{1,n},Z_{2,n})
\stackrel{i.i.d.}{\sim}\mathbb{P}_{\theta_0,\lambda_0,G_0,H_0}\] from which we are only able to observe $Z_{1,j}$, $Z_{2,j}$, $\Delta_j := \ind{T^0_j \leq C_j}$ and $T_j := T^0_j \land C_j$ for $j=1,\ldots, n$. The goal is to estimate the change-point $\tau_0\in(0,1)$ given these observations.

A standard method of estimation in this setting is via Cox's partial likelihood, in which case the likelihood and log-likelihood functions are given by
{\small
\begin{eqnarray}\nonumber
L_n(\tau,\alpha,\beta,\gamma) &:=& \prod_{\begin{subarray}{c} 1\leq k\leq n\\ T_k^0\leq C_k \end{subarray}}\frac{
e^{\alpha\cdot Z_{1,k}(T^0_k) + (\beta+\gamma\ind{T_k^0>\tau})\cdot Z_{2,k}(T^0_k)}}{\sum_{\{1\leq j\leq n:\ T_k^0\leq T_j^0\land C_j \}}e^{\alpha\cdot Z_{1,j}(T^0_k) + (\beta+\gamma\ind{T_k^0>\tau})\cdot Z_{2,j}(T^0_k)}},\\
l_n(\theta)&:=& \function{\log}{L_n(\tau,\xi)} = \function{\log}{L_n(\tau,\alpha,\beta,\gamma)}.\nonumber
\end{eqnarray}
}
In this case, the maximum partial likelihood estimator of the change-point and the covariate multipliers is given by
\begin{equation}\nonumber
\hat{\theta}_n = (\hat{\tau}_n,\hat{\xi}_n) = (\hat{\tau}_n,\hat{\alpha}_n,\hat{\beta}_n,\hat{\gamma}_n) := \sargmax_{\theta\in\Theta} \{ l_n(\theta)\}.
\end{equation}

\cite{pons02} derived the asymptotics for this estimator. For $u=(u^1,u^2,\ldots,u^{1+p+2q})=(u^1,v)\in\mathbb{R}^{1+p+2q}$ define $\theta_{n,u}=\left(\tau_0+\frac{u^1}{n},\xi_0+\frac{v}{\sqrt{n}}\right)$. Then, under some regularity conditions, Theorem 2 in \cite{pons02} shows that
{\small
\begin{equation}\nonumber
\left(n(\hat{\tau}_n - \tau_0),\sqrt{n}(\hat{\xi}_n - \xi_0)\right) = \sargmax_{u\in\mathbb{R}^{1+p+2q}:\ \theta_{n,u}\in\Theta}\{l_n(\theta_{n,u})-l_n(\theta_0)\} = O_\mathbf{P}(1).
\end{equation}
}
It can also be inferred from Proposition 3 and Theorem 3 of the same paper that $\Psi_n:=l_n(\theta_{n,u})-l_n(\theta_0)\rightsquigarrow \Psi$ on $\mathcal{D}_K$ for every compact rectangle $K\subset\mathbb{R}^{1+p+2q}$, where $\Psi$ is a stochastic process of the form
\begin{equation}\label{ec10}
\Psi(u^1,v) = Q(u^1) + v\cdot \tilde{W} - \frac{1}{2}v\tilde{I}\cdot v,
\end{equation}
with $Q$ being a two-sided, compound Poisson process, $\tilde{W}$ a Gaussian random variable independent of $Q$ and $\tilde{I}$ some positive definite matrix on $\mathbb{R}^{(p+2q)\times(p+2q)}$. For a detailed description of $Q$, $\tilde{W}$ and $\tilde{I}$ we refer the reader to Section 4 of \cite{pons02}.

If one defines $\Gamma_n$ and $\Gamma$ to be the pure jump processes associated with $\Psi_n$ and $\Psi$, respectively, it can be shown, using similar techniques as in the proof of Theorem 3 of \cite{pons02}, that $(\Psi_n,\Gamma_n)\rightsquigarrow(\Psi,\Gamma)$ on $\mathcal{D}_B\times\mathcal{D}_{B_1}$ for every compact subinterval $B_1\subset \mathbb{R}$ and compact rectangle $B:=B_1\times B_2\subset \mathbb{R}^{1+p+2q}$. Hence, Theorem \ref{t2} can be applied in this situation to conclude that
\[ \left(n(\hat{\tau}_n - \tau_0),\sqrt{n}(\hat{\xi}_n - \xi_0)\right)\rightsquigarrow\sargmax_{u\in\mathbb{R}^{1+p+2q}}\{\Psi(u)\}.\]
It must be noted that the proof of Theorem 4 in \cite{pons02} makes no mention of the pure jump processes $\Gamma_n$ and $\Gamma$. On the second sentence of this proof, the author claims that the asymptotic distribution follows just from the weak convergence of the processes $\Psi_n$. As we saw in Section \ref{s4s1} this fact alone is not enough to conclude the weak convergence of the smallest maximizers. Thus, the argument given in this section completes the mentioned proof in \cite{pons02}.

\subsection{Estimating a change-point in a Cox regression model according to a threshold in a covariate}\label{s4s3}
We will now discuss another application from survival analysis. Consider again a Cox regression model but now with a covariate process of the form $Z=(Z_1,Z_2,Z_3)$ where $Z_1$ and $Z_2$ are as in Section \ref{s4s2} and $Z_3$ is a continuous random variable in $\mathbb{R}$. We will denote the survival and censoring times  as in Section \ref{s4s2}. We are now concerned with a hazard function of the form
\begin{equation}
\lambda(t|Z) = \lambda(t)e^{\alpha\cdot Z_1(t) + \beta\cdot Z_2(t)\ind{Z_3\leq\zeta} + \gamma\cdot Z_2(t)\ind{Z_3>\zeta}}, \nonumber
\end{equation}
for $\alpha\in\mathbb{R}^q$, $\beta,\gamma\in\mathbb{R}^q$ and some $\zeta\in I$ where $I$ is a closed interval entirely contained in the interior of the support of $Z_3$. We now consider the parameter space $\Theta:=I\times\mathbb{R}^{p+2q}$ and we write $\theta=(\zeta,\xi):=(\zeta,\alpha,\beta,\gamma)\in\Theta$. The partial likelihood and log-likelihood functions are now given by
{\small
\begin{eqnarray}\nonumber
L_n(\zeta,\alpha,\beta,\gamma) &:=& \prod_{\begin{subarray}{c} 1\leq k\leq n\\ T_k^0\leq C_k \end{subarray}}\frac{
e^{\alpha\cdot Z_{1,k}(T^0_k) + \beta\cdot Z_{2,k}(T^0_k)\ind{Z_{3,k}\leq\zeta} + \gamma\cdot Z_{2,k}(T^0_k)\ind{Z_{3,k}>\zeta}}}{\sum_{\{1\leq j\leq n:\ T_k^0\leq T_j^0\land C_j \}}e^{\alpha\cdot Z_{1,j}(T^0_k) + \beta\cdot Z_{2,j}(T^0_k)\ind{Z_{3,j}\leq\zeta} + \gamma\cdot Z_{2,j}(T^0_k)\ind{Z_{3,j}>\zeta}}},\\
l_n(\theta)&:=& \function{\log}{L_n(\zeta,\xi)} = \function{\log}{L_n(\zeta,\alpha,\beta,\gamma)}.\nonumber
\end{eqnarray}
}
As before, we assume that the observations come from a model with some specific value $\theta_0\in\Theta$. Following the notation of Section \ref{s4s2}, for $u=(u^1,u^2,\ldots,u^{1+p+2q})=(u^1,v)\in\mathbb{R}^{1+p+2q}$ define $\theta_{n,u}=\left(\zeta_0+\frac{u^1}{n},\xi_0+\frac{v}{\sqrt{n}}\right)$. Then, under some regularity conditions, Theorem 2 in \cite{pons} shows that
{\small
\begin{equation}\nonumber
\left(n(\hat{\zeta}_n - \zeta_0),\sqrt{n}(\hat{\xi}_n - \xi_0)\right) = \sargmax_{u\in\mathbb{R}^{1+p+2q}:\ \theta_{n,u}\in\Theta}\{l_n(\theta_{n,u})-l_n(\theta_0)\} = O_\mathbf{P}(1).
\end{equation}
}

Lemma 5 and Theorem 3 in \cite{pons} show that $\Psi_n:=l_n(\theta_{n,u})-l_n(\theta_0)\rightsquigarrow \Psi$ on $\mathcal{D}_K$ for every compact rectangle $K\subset\mathbb{R}^{1+p+2q}$, where $\Psi$ is another stochastic process of the form (\ref{ec10}) but with different two-sided, compound Poisson process $Q$, Gaussian random variable $\tilde{W}$ and positive definite matrix $\tilde{I}$. The details can be found in Section 4 of \cite{pons}.

Letting $\Gamma_n$ and $\Gamma$ to be the pure jump processes associated with $\Psi_n$ and $\Psi$, respectively, it can be shown that $(\Psi_n,\Gamma_n)\rightsquigarrow(\Psi,\Gamma)$ on $\mathcal{D}_B\times\mathcal{D}_{B_1}$ for every compact subinterval $B_1\subset \mathbb{R}$ and compact rectangle $B:=B_1\times B_2\subset \mathbb{R}^{1+p+2q}$. Hence, another application of Theorem \ref{t2} shows that
\[ \left(n(\hat{\tau}_n - \tau_0),\sqrt{n}(\hat{\xi}_n - \xi_0)\right)\rightsquigarrow\sargmax_{u\in\mathbb{R}^{1+p+2q}}\{\Psi(u)\}.\]
As in \cite{pons02}, the argument to derive the asymptotic distribution given in the proof of Theorem 5 lacks a proper discussion of the convergence of the associated pure jump processes. Therefore, the analysis just given can be seen as a complement to the proof of Theorem 5 in \cite{pons}.

More general models involving right censoring for survival times and a change-point based on a threshold in a covariate can be found in \cite{koso}. There, the change-point estimator also achieves a $n^{-1}$ rate of convergence. The asymptotic distribution of this estimator also corresponds to the smallest maximizer of a two-sided, compound Poisson process and can be deduced from an application of Theorem \ref{t2}. We would like to point out that the above authors omit a discussion about the associated pure jump processes. They claim the desired stochastic convergence follows from an application of Theorem 3.2.2 in \cite{vw} (see the last paragraph of the proof of Theorem 5 in page 985 of \cite{koso}), but this theorem cannot be applied as the maximizer of a compound Poisson process is not unique. Thus, a proper application of Theorem \ref{t2} would complete the argument in \cite{koso}.

\bibliography{Referencias}
\bibliographystyle{apalike}

\end{document}